\documentclass[12pt,a4paper]{amsart}

\usepackage{amsmath,mathrsfs}
\usepackage{amssymb}
\usepackage{color}
\usepackage{dsfont}
\usepackage[latin1]{inputenc}

\newcommand\NoBlackBoxes{\global\overfullrule0pt}
\NoBlackBoxes

\newcommand{\N}{\mathbb{N}}

\makeatletter
\let\serieslogo@\relax
\let\@setcopyright\relax

\newtheorem{definition}{Definition}[section]
\newtheorem{theorem}[definition]{Theorem}

\newtheorem{rem}[definition]{Remark}

\renewcommand{\P}{{\mathbb{P}}}
\newcommand{\E}{{\mathbb{E}}}
\newcommand{\R}{{\mathbb{R}}}

\newcommand{\s}{\sigma}

\newcommand{\sgn}{\mathrm{sgn}}

\renewcommand{\S}{\mathcal{S}}

\renewcommand{\epsilon}{\varepsilon}
\renewcommand{\phi}{\varphi}
\renewcommand{\a}{\alpha}

\renewcommand{\s}{\sigma}

\renewcommand{\r}{\varrho}

\newcommand{\wt}{\widetilde}

\newcommand{\wht}{\widehat{T}}

\begin{document}

\title[Associative Memory with huge Capacity]{On a model of associative memory with huge storage capacity}
%\titlerunning{Associative Memory with huge Capacity}        % if too long for running head

\author[M. Demircigil]{Mete Demircigil}
\address[Mete Demircigil]{D\'epartement de Math\'ematiques et Applications,
Ecole Normale Sup\'erieure,
45 rue d'Ulm,
75005 Paris,
France}

\email[Mete Demircigil]{mete.demircigil@ens.fr}

\author[J. Heusel]{Judith Heusel}
\address[Judith Heusel]{Fachbereich Mathematik und Informatik,
University of M\"unster,
Einsteinstra\ss e 62,
48149 M\"unster,
Germany}

\email[Judith Heusel]{jheus01@uni-muenster.de}

\author[M. L\"owe]{Matthias L\"owe}
\address[Matthias L\"owe]{Fachbereich Mathematik und Informatik,
University of M\"unster,
Einsteinstra\ss e 62,
48149 M\"unster,
Germany}

\email[Matthias L\"owe]{maloewe@math.uni-muenster.de}

\author[S. Upgang]{Sven Upgang}
\address[Sven Upgang]{Fachbereich Mathematik und Informatik,
University of M\"unster,
Einsteinstra\ss e 62,
48149 M\"unster,
Germany}

\email[Sven Upgang]{s.upgang@uni-muenster.de}

\author[F. Vermet]{Franck Vermet}
\address[Franck Vermet]{Laboratoire de Math\'ematiques de Bretagne Atlantique, UMR CNRS 6205, Universit\'e de Bretagne Occidentale,  6, avenue Victor Le Gorgeu\\
CS 93837\\
F-29238 BREST Cedex 3\\
France}

\email[Franck Vermet]{Franck.Vermet@univ-brest.fr}

\date{\today}

\subjclass[2000]{Primary: 82C32, 60K35, Secondary: 68T05, 92B20}

\keywords{Neural networks, associative memory, Hopfield model, exponential inequalities}

\begin{abstract}
In \cite{KrotovHopfield2016} Krotov and Hopfield suggest a generalized version of the well-known Hopfield model of associative memory. In their version they consider a polynomial interaction function and claim that this increases the storage capacity of the model. We prove this claim and take the ''limit'' as the degree of the polynomial becomes infinite, i.e. an exponential interaction function. With this interaction we prove that model has an exponential storage capacity in the number of neurons, yet the basins of attraction are almost as large as in the standard Hopfield model.

\end{abstract}

%\keywords{Neural networks \and associative memory \and Hopfield model \and exponential inequalities}
%\chapter{Note to Dense Associative Memory for Pattern Recognition

\maketitle

\section{Introduction}
Neural networks and associative memories have been a highly active research area in computer science, physics and probability theory for more than thirty years. The standard model of an associative memory was developed in the seminal paper \cite{Hopfield1982}. His model is based on $N$ neurons, each of which can only take the values $\pm 1$. Each pair of these neurons is connected and thus interacts. We want to store a set of input data $(\xi^\mu)_{\mu=1}^M$, so called patterns or images, where $M$ may and will depend on the system size $N$, in the model. Each of the patterns $\xi^\mu$ is itself a bit string of length $N$, hence $\xi^\mu=(\xi_i^\mu)_{i=1}^N$ where $\xi_i^\mu \in\{-1, +1\}$ for all $i$ and $\mu$.
The strength at which two neurons $i$ and $j$ interact depends on the images and is given by the so-called synaptic efficacy
$$
J_{ij}= \sum_{\mu=1}^M \xi_i^\mu \xi_j^\mu.
$$
With this set of $(J_{ij})$'s we associate a dynamics or updating rule $T=(T_i)_{i=1}^N$ on $\{-1,+1\}^N$ such that
$$
T_i(\sigma):= \sgn(\sum_{j=1}^N J_{ij} \sigma_j) \qquad \s=(\s_i) \in\{-1,+1\}^N
$$
and the indices $i$ are either updated uniformly at random or in a given order. One of the patterns $(\xi^\mu)$ is considered to be stored, if and only if it is stable under the (retrieval) dynamics $T$, i.e. if and only if $T_i(\xi^\mu)=\xi_i^\mu$ for all $i=1, \ldots, N$.

The central question is now: How many patterns can be stored in the above model? Of course, this sensitively depends on the way we choose these patterns.
Much in agreement with the choice of messages in information theory in most of the test scenarios for associative memories the patterns are chosen independent and identically distributed (even though, other choices may be considered as well, see e.g. \cite{Lo98}, \cite{Lo99a}, \cite{LV05} or \cite{LV07}). More precisely, it is assumed that
$$\P(\xi_i^\mu=1)= \P(\xi_i^\mu=-1)=\frac 12 \quad \mbox{for all $i$ and $\mu$}$$
and that all $(\xi_i^\mu)_{i,\mu}$ are i.i.d.
Under these assumptions it was shown in \cite{MPRV} that the Hopfield model described above can store $M= C \frac N {\log N}$ patterns with $C<\frac 12$, if we require that a fixed (but arbitrary) pattern is stable, and $C <\frac 14$ if we ask for stability of all patterns simultaneously (also see \cite{Bov98} for a matching upper bound).
However, non-rigorous computations involving the so-called replica trick, suggest that we may even achieve a capacity of $M= \alpha N$ if we allow for a small fraction of errors in the retrieval and $\alpha < 0.138$ (see \cite{AGS}, \cite{AGS2}). This prediction was mathematically confirmed (yet with smaller values of $\alpha$) in \cite{Newman_hopfield},  \cite{loukianova}, and \cite{talagrand}.

In a very recent contribution Krotov and Hopfield  suggest a generalized version of the Hopfield model (see \cite{KrotovHopfield2016}). There, the authors replace the updating dynamics $T=(T_i)$ by a more general, still asynchronous one, namely:
%
%We use the standard model of associative memory with $N$ binary neurons taking values in $\{-1, 1\}$. We want the memory to store patterns $(\xi^\mu)_{\mu=1,\ldots, m}$ chosen uniformly at random. We proclaim that the neural net has a memory capacity of exponential order ($\exp(N^\b)$ for any $\b < 1$). Additionally the net can restore these patterns if they have $\p N$ randomly chosen errors for a fixed $\p \in [0,0.5)$.  \\
%
%A configuration of the network is denoted by the vector $\s = (\s)_{i=1,\ldots, N}$. The paper \cite{Hopfield2016a} suggests to use the following update rule for a non-decreasing function $f$: \\
%In an asynchronous update the neuron $i$ will flip to the value
\begin{align}\label{dyn_gen}
T_i(\s) :=
\sgn\Bigl[
\sum\limits_{\mu=1}^M
\bigl( F\bigl( 1 \cdot \xi^\mu_i + \sum\limits_{j\neq i} \xi^\mu_j \sigma_j\bigr)
- F\bigl( (-1) \cdot \xi^\mu_i + \sum\limits_{j \neq i} \xi^\mu_j \sigma_j\bigr) \bigr) \Bigr]
\end{align}
where $F: \R \to \R$ is some smooth function. The case of $F(x) = x^2$ reduces to the standard Hopfield model with the quadratic crosstalk-terms introduced above. Indeed, in this case the argument in the sign-function is given by the difference
\begin{eqnarray*}
&&\sum\limits_{\mu=1}^M F\bigl( 1 \cdot \xi^\mu_i + \sum\limits_{j\neq i} \xi^\mu_j \sigma_j\bigr)
- F\bigl( (-1) \cdot \xi^\mu_i + \sum\limits_{j \neq i} \xi^\mu_j \sigma_j\bigr)\\
& = &
\sum\limits_{\mu=1}^M 1+ 2\sum\limits_{j\neq i} \xi_i^\mu \xi^\mu_j \sigma_j+ (\sum\limits_{j\neq i} \xi^\mu_j \sigma_j)^2-
1+ 2\sum\limits_{j\neq i} \xi_i^\mu \xi^\mu_j \sigma_j-(\sum\limits_{j\neq i} \xi^\mu_j \sigma_j)^2\\
&=& 4\sum\limits_{\mu=1}^M\sum\limits_{j\neq i} \xi_i^\mu \xi^\mu_j \sigma_j
\end{eqnarray*}
and its sign is of course the same as that of $\sum\limits_{\mu=1}^M\sum\limits_{j\neq i} \xi_i^\mu \xi^\mu_j \sigma_j$, therefore the dynamics are equal.

The reason for this more general choice of the interaction function $F$ is the following insight: In the standard Hopfield model there is an energy associated with the dynamics $T$ (i.e. the energy of a configuration decreases by an application of $T$).  This energy of a spin configuration $\s$ is given by $H(\s)= -\frac 1N\sum_\mu \sum_{i,j} \xi_i^\mu \xi_j^\mu \s_i \s_j$ and it
decreases ''too slowly'' as the configuration $\s$  approaches pattern to allow for a superlinear storage capacity

The authors in \cite{KrotovHopfield2016} therefore in particular analyze what they call ''polynomial interaction'' (or, as they put it, energy) functions, i.e. $F(x)=x^n$. Krotov and Hopfield state the following assertion
\begin{theorem}\label{theoKH} [\cite{KrotovHopfield2016}, formulas (5) and (6)]
\begin{enumerate}
\item
In the generalized Hopfield model with interaction function $F(x)=x^n$ one can store up to $M=\alpha_n N^{n-1}$ patterns, if small errors in the retrieval are tolerated.
\item In the same model, one can store $M= \frac{N^{n-1}}{c_n \log N}$ patterns for $c_n > 2 \,(2n-3)!!$, if one wants a fixed pattern to be a fixed point of the dynamics $T$
introduced in \eqref{dyn_gen} with a probability converging to 1.
\end{enumerate}
\end{theorem}
A proof of this theorem could probably be rather involved. This is due to the fact that the energy function of
the model described in Theorem \ref{theoKH} is of a polynomial form.
As a matter of fact, up to normalization, the energy of the model in these cases is $H(\s)= \sum_\mu P(m^\mu)$, where $m^\mu:= \sum_i \xi_i^\mu \s_i$ is the overlap of the configuration $\s$ with the $\mu$'th pattern and $P$ is a polynomial, or, in other words, with $F(x)=x^n$ and $n$ even
$$
F\left( 1 \cdot \xi^\mu_i + \sum\limits_{j\neq i} \xi^\mu_j \sigma_j\right)
- F\left( (-1) \cdot \xi^\mu_i + \sum\limits_{j \neq i} \xi^\mu_j \sigma_j\right)
$$
consists of many summands of the form $\xi_i^\mu (\sum\limits_{j \neq i} \xi^\mu_j \sigma_j)^m$ where $m$ is an even number smaller than $n$.
There is, however, a closely related model, where the above statement can be proven.
To this end consider the dynamics $\widehat{T}=(\wht_i)$ on $\{-1,+1\}^N$ defined by
\begin{equation}\label{multdyn}
\wht_i(\sigma):= \mathrm{sgn}\left(\sum_{1=j_1, \ldots j_{n-1}}^N
\sigma_{j_1} \cdots \sigma_{j_{n-1}}  W_{i,j_1 \ldots
j_{{n-1}}}\right).
\end{equation}
with
\begin{equation}\label{multhebb}
W_{i_1, \ldots, i_n}= \frac 1 {N^{n-1}} \sum_{\mu=1}^M
\xi_{i_1}^\mu \xi_{i_2}^\mu \cdots \xi_{i_n}^\mu.
\end{equation}
Then the following theorem can be shown
\begin{theorem}\label{theo nspin}
\begin{enumerate}
\item
In the generalized Hopfield model with dynamics $\wht$ defined in \eqref{multdyn} and \eqref{multhebb} one can store up to $M=\alpha_n N^{n-1}$ patterns, if small errors in the retrieval are tolerated.
\item In the same model, one can store $M= \frac{N^{n-1}}{c_n \log N}$ patterns for $c_n > 2 \,(2n-3)!!$, if one wants a fixed pattern to be a fixed point of the dynamics $\wht$ with a probability converging to 1.
\end{enumerate}
\end{theorem}

While part 1 of the above theorem was proved in \cite{Newman_hopfield}, we will give a proof of the second part (including a computation of the constants $c_n$) in Section 2 below. Note that the thermodynamics of this model was analyzed in \cite{BovierNiederhauser}.

More interesting than these polynomial models is, however, the question what happens if we formally send $n$ to infinity. One could conjecture from above that this would lead to an interaction function of the form $e^{x}$, on the one hand and to a super-polynomial storage capacity, on the other. This is indeed what we will show.
Actually, we will even prove slightly more: In general, one could imagine that an increase in capacity goes to expense of associativity, such that in the extreme case, one could store $2^N$ patterns but none of them has a positive radius of attraction. We will show that this is not the case for our model: The dynamics is even able to repair an amount of random errors of order $N$.

\begin{theorem}\label{theo_exp}
Consider the generalized Hopfield model with the dynamics described in \eqref{dyn_gen} and interaction function $F$ given by $F(x)= e^{x}$.
For a fixed $0 < \alpha < \log(2)/2$ let $M=\exp\left(\alpha N  \right)+1$ and let $\xi^1, \ldots, \xi^M$ be $M$ patterns chosen uniformly at random from $\{-1,+1\}^N$. Moreover fix $\r \in [0, 1/2)$. For any $\mu$ and any $\wt \xi^\mu $ taken uniformly at random from the Hamming sphere with radius $\r N$ centered in $\xi^\mu$, $\S(\xi^\mu,\r N)$, where $\r N$ is assumed to be an integer, it holds that
%we have
%\begin{align*}\P\left( \exists \mu \; \exists i  : T_i\left(\wt \xi^\mu \right) \neq \xi^\mu_i \right) \rightarrow 0
%\end{align*}
%as $N \rightarrow \infty$.\\
%For $\beta=1$ we choose $M=\exp\left(\alpha N  \right)$. We have
$$\P\left( \exists \mu \; \exists i  : T_i\left(\wt \xi^\mu \right) \neq \xi^\mu_i \right) \rightarrow 0,$$
if $\alpha$ is chosen in dependence of $\r$ such that
%$$2\alpha-(1-2\r-\alpha)\cdot\text{artanh}(1-2\r-\alpha)+\log\left[\text{cosh}\left(\text{artanh}(1-2\r-\alpha)\right)\right]<0
%$$

$$ \a < \frac{I(1-2\r)}{2}
$$

with $I:x \mapsto \frac{1}{2}\left((1+x) \log (1+x) + (1-x) \log(1-x)\right)$.
\end{theorem}

\begin{rem}
\normalfont
Note that Theorem \ref{theo_exp} in particular implies that
\begin{align*}
\P\left( \exists \mu \; \exists i  : T_i\left(\xi^\mu \right) \neq \xi^\mu_i \right) \rightarrow 0
\end{align*}
as $N \rightarrow \infty$, i.e. with a probability converging to 1, all the patterns are fixed points of the dynamics. In this case we can even take
$\a < \frac{I(1)}{2}$.
\end{rem}

\begin{rem}
	\normalfont
	Theorem \ref{theo_exp} can be proven analogously if the configuration $\wt \xi^\mu$ is drawn uniformly at random from a Hamming ball $B(\xi^{\mu}, \rho N)$. Indeed, the probability of correcting an arbitrary pattern of the sphere $S(\xi^\mu, \rho N)$ can be used as a bound for the probability of correcting an arbitrary pattern of lower spheres. 
\end{rem}

Theorem \ref{multdyn} and Theorem \ref{theo_exp} will be proven in the following section.

\section{Proofs}
%
%The case of $F(x) = x^2$ represents the standard Hopfield model with the quadratic crossterms. In these models the neurons always flip to the configuration with lower energy and this is basically what the update rule states.\\
We start with the proof of Theorem \ref{theo nspin}.

\begin{proof}[Proof of Theorem \ref{theo nspin}]\leavevmode\newline
As already mentioned the first part of the Theorem has already been proven in \cite{Newman_hopfield}.

We are interested in bounding the following probability
\begin{align*}
\P(\exists i\le N: \; \widehat{T}_i(\xi^1) \ne \xi^1_i)
& = \P\Bigl(\exists i\le N: - \sum\limits_{\mu = 2}^M \xi_i^1 \xi_i^\mu \; \bigl(\sum\limits_{j} \xi_{j}^1 \xi_{j}^\mu \bigr)^{n-1} > N^{n-1} \Bigr).
\end{align*}
With the exponential Chebyshev inequality and the independence of the random variables $(\xi_i^\mu)$ for different $\mu$ it follows that
\begin{align*}
%\P&
\P(\exists i\le N: \; \widehat{T}_i(\xi^1) \ne \xi^1_i) %\\
%&
\le N e^{-t N^{n-1}} \E\Bigl[\exp\bigl(-t \, \xi^1_i \xi^2_i \bigl(\sum\limits_{j} \xi^1_j \xi^2_j \bigr)^{n-1}\bigr) \Bigr]^{M-1}.
\end{align*}
For the moment generating function we condition on the values of $\xi^1_i \xi^2_i$ to get the upper bound
\begin{align*}
& \P(\exists i\le N: \; \widehat{T}_i(\xi^1) \ne \xi^1_i) \le N e^{-t N^{n-1}} \E\big[\cosh(t (\sum\limits_{j}  \xi^1_j \xi^2_j)^{n-1} ) \big]^{M-1}.
\end{align*}
The random variables $(\xi^1_j \xi^2_j)_{j}$ are i.i.d. and distributed like $\xi^1_1$. Define $m = \frac{1}{\sqrt{N}} \sum_{j} \xi^1_j \xi^2_j$ and write the expectation as the sum over all possible values $x$ of $m$:
\begin{align*}
\E\big[\cosh(t (\sum\limits_{j}  \xi^1_j \xi^2_j)^{n-1} ) \big]
& = \sum\limits_{x} \cosh\Big(t N^{\frac{n-1}{2}} x^{n-1} \Big) \cdot  \P( m = x).
\end{align*}
The sum is over all $x\in\{0, \pm \frac{1}{\sqrt{N}}, \ldots, \pm \sqrt{N}\}$. First we want to eliminate the tail events and use the fact that the probability vanishes fast enough if we restrict $x$ away from zero. To this end we fix $\beta > \frac{1}{2}$ and split the sum at $\log(N)^\beta$. Additionally observe that $x$ cannot grow faster than $\sqrt{N}$ and set $t = a_n/M$ for $a_n > 0$ independent of $N$. Thus
\begin{align*}
& \sum\limits_{x: \log(N)^\beta < |x| \le \sqrt{N}} \cosh\left( t N^{\frac{n-1}{2}} x^{n-1} \right) \cdot \P(m = x)\\
 \le & 2 \cosh\left( t N^{n-1} \right) \P(m > \log(N)^\beta)
\le  2 \; \exp\left( t N^{n-1} \right) \exp\left(-\frac{1}{2} \log(N)^{2\beta} \right)\\
 = & 2 \; \exp\left( \left[ a_n c_n -\frac{1}{2} \log(N)^{2\beta - 1} \right] \log(N) \right).
\end{align*}
 Here we used a standard large deviations estimate and $\cosh(z) \le \exp(|z|)$.  This part of the moment generating function converges to zero for $N \rightarrow \infty$ because for $N$ large enough the term in brackets can be bounded by a negative value.

For the critical values of $m$ we use the inequality $\cosh(z) \le \exp(\frac{z^2}{2})$ for all $z$ and write the exponential function in its Taylor expansion:
\begin{align*}
& \sum\limits_{x: |x| \le \log(N)^\beta} \cosh\left( t N^{\frac{n-1}{2}} x^{n-1} \right)\P(m = x)
%\\&
\le  \sum\limits_{x:|x| \le \log(N)^\beta}  e^{\frac{t^2}{2} N^{n-1} x^{2(n-1)}}\P(m = x) \\
& = \sum\limits_{x: |x| \le \log(N)^\beta} \left(1 + \frac{t^2}{2} N^{n-1} x^{2(n-1)} +  \sum\limits_{k=2}^\infty \frac{1}{2^k} \frac{(t^2 N^{n-1} x^{2(n-1)})^k}{k!}  \right) \P(m = x).
\end{align*}
The distribution of $m$ converges to a standard normal distribution and its moments are bounded by the moments of the latter. For $l\in \N$ let $\kappa_{2l}$ be the $2l$-th moment of a standard normal distribution. The sum of probabilities can be bounded by one. So for $N$ large enough, $t= a_n/M$, and $M = \mathrm{const.} \frac{N^{n-1}}{\log N}$ we derive the following upper bound for the moment generating function:
\begin{align*}
& \sum\limits_{x: |x| \le \log(N)^\beta} \cosh\left( t N^{\frac{n-1}{2}} x^{n-1} \right)\P(m = x)   \\
&\le   1+ \frac{t^2}{2} N^{n-1} \kappa_{2(n-1)} + \sum\limits_{x:\; |x| \le \log(N)^\beta} \sum\limits_{k=2}^\infty \frac{1}{2^k} \frac{\left(t^2 N^{n-1} x^{2(n-1)}\right)^k}{k!} \cdot \P(m = x) \\
& \le  1+ \frac{t^2}{2} N^{n-1} \kappa_{2(n-1)}+ \sum\limits_{k=2}^\infty \frac{1}{2^k} \frac{(t^2 N^{n-1} \log(N)^{2 \beta (n-1)})^k}{k!} \\
%\end{align*}
%\begin{align*}
& \le 1+ \frac{t^2}{2} N^{n-1} \kappa_{2(n-1)}+ \frac{t^4}4 N^{2n-2} \log(N)^{4 \beta (n-1)}(e-2)\\
& \le \exp\left(\frac{t^2}{2} N^{n-1} \kappa_{2(n-1)} + t^4 N^{2(n-1)} \log(N)^{4\beta (n-1)} \right)
\end{align*}
Inserting $t= a_n/M$ into this result we obtain
\begin{align*}
\P&(\exists i\le N: \; \widehat{T}_i(\xi^1) \ne \xi^1_i) \\
& \le \exp\left(\log(N)-t N^{n-1}\right) \exp\left(\frac{t^2}{2} N^{n-1} \kappa_{2(n-1)} M + t^4 N^{2(n-1)} \log(N)^{4\beta (n-1)} M \right) \\
& = \exp \left( \left[1 - a_n c_n \left(1 - \frac{a_n \, \kappa_{2(n-1)}}{2} \right) \right] \log(N)  \right) \cdot (1+o(1)).
\end{align*}
for our choice of $t$ and $M$.
%\begin{align*}
%t^4 N^{2(n-1)} \log(N)^{4\beta (n-1)} M
%& = \frac{a_n^4 c_n^3 \log(N)^{4\beta (n-1) + 3}} {N^{n-1}} \rightarrow 0
%\end{align*}
%for $N$ going to infinity. Thus only the first two exponential function have an impact for the upper bound.

The moments of the standard normal distribution is given by $\kappa_{2(n-1)} = (2n-3)!!$ for all $n\in \N$. Choose $a_n = (\kappa_{2(n-1)})^{-1}$. The term in brackets can be bounded by negative value if $c_n$ satisfies
$$
1- a_n c_n \left( 1 - \frac{ a_n \, \kappa_{2(n-1)}}{2} \right) < 0
$$
which is the case if and only if
$c_n>2 (2n-3)!!$.
This is exactly the memory capacity proposed by Hopfield and Krotov in Theorem \ref{theoKH}. \\
\end{proof}

We continue with the proof of our central result.
\begin{proof}[Proof of Theorem \ref{theo_exp}]\leavevmode\newline
We start by slightly reformulating the dynamics of the model. Indeed,
an (almost) equivalent formulation is  to say that a neuron $i$ will remain unchanged after an application of the update rule if
\begin{align}
\label{updaterule}
\Delta_i E(\sigma) :=
\sum\limits_{\mu=1}^M
\left( F\left( \s_i \xi^\mu_i + \sum\limits_{j\neq i} \xi^\mu_j \sigma_j\right)
- F\left( -\s_i \xi^\mu_i + \sum\limits_{j \neq i} \xi^\mu_j \sigma_j\right) \right) > 0
\end{align}
and the spin of neuron $i$ will be changed if $\Delta_i E(\s)$ is less than zero. In the limit $N\rightarrow \infty$ the case $\Delta_i E(\s) = 0$  is negligible for the later purposes.

Starting in one of the patterns (without loss of generality the first one $\xi^1$) we want to show that it is an attractive fixed point of the update dynamics, i.e. we need to show that $T_i(\xi^\mu)=\xi_i^\mu$ and we want the model to correct up to $\r N$ random errors by updating each of the neurons once.  Without loss of generality we focus on the pattern $\xi^1$, taking a corrupted version $\tilde \xi^1$  uniformly at random from $S(\xi^1, \rho N)$ for a fixed $\rho \in [0, \frac 12)$, and the neuron $i$. Then we can interpret the summand for $\mu = 1$ in (\ref{updaterule}):
\begin{align*}
E_\text{signal} = F\left(\sum\limits_{j=1}^N \xi^1_j \wt \xi^1_j \right) - F\left(-2 \, \xi^1_i \wt \xi^1_i + \sum\limits_{j=1}^N \xi^1_j \wt \xi^1_j \right)
\end{align*}
as signal term and the rest of the sum in (\ref{updaterule}):
\begin{align*}
E_\text{noise} = \sum\limits_{\mu=2}^M
\left( F\left( \sum\limits_{j=1}^N \xi^\mu_j \wt \xi_j\right)
- F\left( -2\wt \xi_i \xi^\mu_i + \sum\limits_{j =1}^N \xi^\mu_j \wt \xi_j\right) \right)
\end{align*}
as noise term. As we will show, in order to have neuron $i$ not updated correctly, the noise term needs to have a bigger impact in (\ref{updaterule}) than the signal term. We want to show that the probability for this event vanishes for $N \rightarrow \infty$.

We need to distinguish two cases: First the neuron $i$ can be correct, i.e. $\xi_i^1= \wt \xi_i^1$, and we want the associative memory not to change this value (this means $\Delta_i E(\wt \xi^1) > 0$). In the other case the neuron is wrong, i.e. $\xi_i^1 = -\wt \xi_i^1$, and the network needs to correct this neuron (this means $\Delta_i E(\wt \xi^1)<0$). In both cases the signal term supports the desired behavior. Indeed, we have:
\begin{align*}
& F\left(\sum_{j=1}^N \xi^1_j \wt \xi^1_j \right) - F\left( \sum_{j=1}^N \xi^1_j \wt \xi^1_j  -2 \right) > 0 \text{ on the one hand, and }\\
& F\left(\sum_{j=1}^N \xi^1_j \wt \xi^1_j \right) - F\left( \sum_{j=1}^N \xi^1_j \wt \xi^1_j  +2 \right) < 0 \text{ on the other }
\end{align*}
depending on whether neuron $i$ is correct or incorrect (since $\sum_{j=1}^N \xi^1_j \wt \xi^1_j \geq (1-2\varrho)N> 2$ and $F=\exp$). This means that in order for neuron $i$ to update correctly, it must be that $\sgn\left(E_\text{signal}+E_\text{noise}\right)=\sgn\left(E_\text{signal}\right)$, which is fulfilled as soon as $ |E_\text{noise}| < |E_\text{signal}|$.
Thus, a necessary condition for the event $\{T_i(\wt \xi^1) \neq \xi^1_i\}$ is that
$  | E_\text{noise} | \geq | E_\text{signal} |.
$
Therefore:
\begin{align*}
& \P\left( \exists \mu \; \exists i  : T_i\bigl(\wt \xi^\mu \bigr) \neq \xi^\mu_i \right)
\leq  N \cdot M \cdot \P\bigl(T_i(\wt \xi^1) \neq \xi^1_i \bigr) \\
 \leq &   N \cdot M \cdot \P\bigl(|E_\text{noise}| \ge |E_\text{signal}|\bigr).
\end{align*}
By using the straightforward fact that $|e^{a\pm 1}-e^{a\mp 1}| \leq [1-e^{-2}] e^2 \cdot e^{a}$:
\begin{align*}
|E_\text{noise}|& \leq \sum\limits_{\mu=2}^M \left|\exp(\xi^\mu_i \xi^1_i + \sum\limits_{j\neq i} \xi^\mu_j \wt \xi^1_j)- \exp(-\xi^\mu_i \xi^1_i + \sum\limits_{j\neq i} \xi^\mu_j \wt \xi^1_j)\right|\\
&\leq [1-e^{-2}]e^2\sum\limits_{\mu=2}^M \exp\left(\langle \xi^\mu | \wt \xi^1 \rangle\right)
\end{align*}
where $\langle x | y \rangle$ is the inner product on $\{\pm 1\}^N$. At the same time $$| E_\text{signal} | > e^{N(1-2\r)}[1-e^{-2}].$$ It follows that:
\begin{align}\label{central}
\P\left( \exists \mu \; \exists i  : T_i\bigl(\wt \xi^\mu \bigr) \neq \xi^\mu_i \right)
&\leq N \cdot M \cdot \P\bigl(T_i(\wt \xi^1) \neq \xi^1_i \bigr) \nonumber \\
&\leq N \cdot M \cdot \P\left( \sum\limits_{\mu=2}^M e^2\exp(\langle \xi^\mu | \wt \xi^1 \rangle) >  e^{N(1-2\r)} \right).
\end{align}
We will use two standard estimates from the theory of large deviations \cite{dembozeitouni}: For a Binomially distributed random variable $\bar{S}_{m,p}$ with parameters $m$ and $p$
%we have
%Let $S_n$ be a sum of $n$ independent Bernoulli variables of mean $p$. Then
and for $\epsilon>0$, we have:
\begin{align}\label{corollary1}
\P\left( \bar{S}_{m,p} \geq m(p+\epsilon) \right) \leq \exp \left( -m \frac{\epsilon^2}{2(p+\epsilon)} \right)
\end{align}
and for a sum $S_m$ of $m$ i.i.d. random variables $X_i$ with $\P(X_1=1)=\P(X_1=-1)=\frac{1}{2}$ and $x\in(0,1)$ we have
\begin{align}\label{corollary2}
\P\left( S_m \geq m x\right) \leq \exp \left( - m I(x) \right)
\end{align}
as well as
\begin{align}\label{corollary3}
\lim\limits_{m\rightarrow\infty} \frac{1}{m} \log(\P\left(S_m > mx\right)) = - I(x)
\end{align}
where again $ I(x)=\frac{1}{2} \left( (1+x) \log (1+x) +(1-x) \log (1-x) \right)$. In fact, \eqref{corollary3} is nothing but Cram\'ers theorem for fair, $\pm 1$-Bernoullis. 

%We will use the following lemma and its corollaries:
%\begin{lemma} Let $S_n$ be a sum of $n$ independent Bernoulli variables of mean $p$. Then for $\epsilon>0$, we have:
%\begin{align*}
%\P\left( S_n \geq n(p+\epsilon) \right) \leq \exp \left( -n D(p+\epsilon || p) \right)
%\end{align*}
%where $D(p||q)=p \log \frac{p}{q} + (1-p) \log \frac{1-p}{1-q}$
%\end{lemma}
%\begin{corollary}\label{corollary1}
%Let $S_n$ be a sum of $n$ independent Bernoulli variables of mean $p$. Then for $\epsilon>0$, we have:
%\begin{align*}
%\P\left( S_n \geq n(p+\epsilon) \right) \leq \exp \left( -n \frac{\epsilon^2}{2(p+\epsilon)} \right)
%\end{align*}
%\end{corollary}
%\begin{corollary}\label{corollary2}
%Let $(X_i)_{i \in \{ 1, \ldots, n}$ independent identically distributed variables such that $\P(X_1=1)=\P(X_1=-1)=\frac{1}{2}$. Let $S_n$ be the sum of those variables. Then for $x>0$, we have:
%\begin{align*}
%\P\left( S_n \geq nx\right) \leq \exp \left( -n I(x) \right)
%\end{align*}
%where $ I: x \mapsto \frac{1}{2} \left( (1+x) \log (1+x) +(1-x) \log (1-x) \right)$
%\end{corollary}
Now let $\a < \frac{1}{2}I(1-2\r)$, $M = \exp \left(\a N\right) + 1$ and $\beta_o$ be such that $I(\beta_o)=\a$. By continuity of $I$ there exists an $\epsilon>0$, such that for all $x\in (1-2\r-\epsilon,1-2\r]$ we have that $\a<\frac{1}{2}I(x)\leq\frac{1}{2}I(1-2\r)$. Again by continuity of $I$ we can choose $\beta<\beta_o$ such that:
$$\a-\frac{\epsilon}{2}=I(\beta_o)-\frac{\epsilon}{2}<I(\beta)< I(\beta_o)=\a.
$$
Let us define
\begin{eqnarray*}
A&=& \{ \mu \in \{2 \ldots M \} | \langle \xi^\mu | \wt \xi^1\rangle \geq \beta N \}, \\
p&=&\P \left( \langle \xi^2 |\wt \xi^1\rangle  \geq \beta N\right).
\end{eqnarray*}
By \eqref{corollary2}, we have that: $p < \exp(-N I(\beta) )$. On the other hand, by \eqref{corollary3} we can conclude that for $\eta = \frac{1}{2} (\alpha-I(\beta)) > 0$ and $N$ sufficiently large we have $p> \exp(-N(I(\beta)+\eta))$.
%\rot{Here reviewer one has a point that both inequalities can hardly be satisfied at the same time ... what do we mean here? Also we should explain, to which term we apply Stirling. Sven: Changed it. Maybe this is ok? I added equation 8. This should be in the book too (?!) but I didn't check it yet.}
%On the other hand, by an application of Stirling's Formula, for $\eta>0$ and $N$ sufficiently large we have $p> \exp(-N(I(\beta)+\eta))$.
%\rot{Here reviewer one has a point that both inequalities can hardly be satisfied at the same time ... what do we mean here? Also we should explain, to which term we apply Stirling}

Now we compute the probability in \eqref{central} by picking those patterns $\xi^\mu$ with a significant overlap with $\xi^1$ (\textit{i.e.} $\mu \in A$).
\begin{align*}
 &\P\left( \sum\limits_{\mu=2}^M e^2\exp(\langle \xi^\mu | \wt \xi^1 \rangle) >  e^{N(1-2\r)} \right)\\
=& \sum_{X\subset \{2 \ldots M \}}\P\left( \left(\sum\limits_{\mu=2}^M \exp(\langle \xi^\mu | \wt \xi^1 \rangle) > e^{N(1-2\r )-2}\right) \cap (A=X)\right)\\
\leq& \sum_{k=0}^{M-1}\sum_{X\in P_k (\{2 \ldots M \})}p^k (1-p)^{M-1-k}\cdot\\
&\P\left( \sum\limits_{\mu \in X}\exp(\langle \xi^\mu | \wt \xi^1 \rangle) + (M-1-k) e^{\beta N} > e^{N(1-2\r )-2}  \middle| A=X\right)
\end{align*}
where $P_k$ denotes the subsets of size $k$. Additionally we used that every overlap in $A^c$ can be bounded by $\exp(\beta N)$. By the identical distribution of the patterns this is equal to
\begin{align*}
=& \sum_{k=0}^{M-1}\binom{M-1}{k}p^k (1-p)^{M-1-k}\cdot \\
&\P\left( \sum\limits_{\mu =2}^{k+1} \exp(\langle \xi^\mu | \wt \xi^1 \rangle) > e^{N(1-2\r) -2} -(M-1-k) e^{\beta N} \middle| A=\{2, \ldots, k+1\}\right).
\end{align*}
By using the maximal summand as an upper bound, a standard union bound and the identical distribution for all $\mu$ we arrive at the following term
\begin{align*}
\leq& \sum_{k=0}^{M-1}\binom{M-1}{k}p^k (1-p)^{M-1-k}\cdot\\
&\P\left( \max\limits_{\mu \in \{2,\ldots, k+1\}} \exp(\langle \xi^\mu | \wt \xi^1 \rangle) > \frac{1}{k}( e^{N(1-2\r)-2} -(M-1-k) e^{\beta N})\middle| A=\{2, \ldots, k+1\}\right) \\
%\leq& \sum_{k=0}^{M-1}\sum_{\mu =2}^{k+1}\binom{M-1}{k}p^k (1-p)^{M-1-k}\cdot \\
%&\P\left( \exp(\langle \xi^\mu | \wt \xi^1 \rangle) > \frac{1}{k}( e^{N(1-2\r)-2} -(M-1-k) \exp(\beta N))| A=\{2, \ldots, k+1\}\right)\\
\leq& \sum_{k=0}^{M-1}k\binom{M-1}{k}p^k (1-p)^{M-1-k}\cdot \\
&\P\left( \exp(\langle \xi^2 | \wt \xi^1 \rangle) > \frac{1}{k}( e^{N(1-2\r)-2} -(M-1-k) e^{\beta N})\middle| A=\{2, \ldots, k+1\}\right).
\end{align*}
%=& \sum_{k=0}^{M-1}k\binom{M-1}{k}\cdot \\
%&\P\left( \left(\exp(\langle \xi^2 | \wt \xi^1 \rangle) > \frac{1}{k}( e^{N(1-2\r)-2} -(M-1-k) \exp(\beta N))\right); A=\{2, \ldots, k+1\}\right)
%\leq &\P\left(\max\limits_{\mu \neq 1} \left[ \exp\left( -\xi^\mu_i \xi^1_i + \sum\limits_{j\neq i} \xi^\mu_j \wt \xi^1_j \right) - \exp\left( \xi^\mu_i \xi^1_i + \sum\limits_{j\neq i} \xi^\mu_j \wt \xi^1_j \right) \right] > M^{-1} \, e^{N(1-2\r)}[1-e^{-2}] \right) \\
%\leq &M \cdot \P\left( \Bigg| \exp\left( -\xi^\nu_i \xi^1_i + \sum\limits_{j\neq i} \xi^\nu_j \wt \xi^1_j \right) - \exp\left( \xi^\nu_i \xi^1_i + \sum\limits_{j\neq i} \xi^\nu_j \wt \xi^1_j \right) \Bigg| > M^{-1} e^{N(1-2\r)}[1-e^{-2}] \right) \\
%= &M \cdot \P\left( 2 \cdot \sinh(1)  \exp\left(\sum\limits_{j\neq i} \xi^\nu_j \wt \xi^1_j \right) > M^{-1} e^{N(1-2\r)}[1-e^{-2}] \right) \\
%&= M \cdot \P\left( \exp\left(\sum\limits_{j\neq i} \xi^\nu_j \wt \xi^1_j \right) > M^{-1} e^{N(1-2\r)} \frac{[1-e^{-2}]}{2 \, \sinh(1)} \right) \\
%&= M \cdot \P\left( \sum\limits_{j\neq i} \xi^\nu_j \wt \xi^1_j > \log\left(M^{-1} e^{N(1-2\r)} \frac{1}{e}\right) \right) \\
%&= M \cdot \P\left( \sum\limits_{j\neq i} \xi^\nu_j \wt \xi^1_j > - \log(M) + N (1-2\r) - 1 \right)
%\end{align*}
Denote by
$$r = r(k)= \P \left( \exp\left(\langle \xi^2 | \wt \xi^1 \rangle\right) \geq \frac{1}{k}( e^{N(1-2\r)-2} -(M-1-k) e^{\beta N})\right).$$
We then arrive at
\begin{align*}
& \P\left( \left(\exp(\langle \xi^2 | \wt \xi^1 \rangle) > \frac{1}{k}( e^{N(1-2\r)-2} -(M-1-k) e^{\beta N})\right) \middle| A=\{2, \ldots, k+1\}\right) \\
& = \frac{\P\left(\left(\exp(\langle \xi^2 | \wt \xi^1 \rangle) > \frac{1}{k}( e^{N(1-2\r)-2} -(M-1-k) e^{\beta N})\right); 2 \in A \right)}{\P\left( 2 \in A \right)} %\\
%& \le \frac{\P\left( \left(\exp(\langle \xi^2 | \wt \xi^1 \rangle) > \frac{1}{k}( e^{N(1-2\r)-2} -(M-1-k) \exp(\beta N))\right)\right)}{p} \\
%&
\le \frac{r(k)}{p}
\end{align*}
because $\P\left(2 \in A\right) = p$.

Thus an upper bound is:
\begin{align*}
&\P\Bigl( \sum\limits_{\mu=2}^M \exp(-\xi_i^\mu \xi^1_i + \sum\limits_{j\neq i} \xi_j^\mu \wt \xi^1_j)- \exp(\xi_i^\mu \xi^1_i + \sum\limits_{j\neq i} \xi_j^\mu \wt \xi^1_j)> c e^{N(1-2\r)}\Bigr)\\
\le&
 \sum_{k=0}^{M-1}k\binom{M-1}{k} r(k) p^{k-1}(1-p)^{M-1-k}.
\end{align*}
Now let us split this sum into two parts: the first one for $k \in \{0, \ldots , \lfloor 2p(M-1) \rfloor\}$ and the second one the remaining part.
We start with the second part. By using the identity $k\binom{M-1}{k}=(M-1)\binom{M-2}{k-1}$ and the trivial fact that $r(k)\leq 1$, we get:
\begin{align*}
 &\sum_{k=\lfloor 2p(M-1)\rfloor +1}^{M-1} k\binom{M-1}{k} r(k) p^{k-1}(1-p)^{M-1-k}\\
\leq& (M-1)\sum_{k=\lfloor 2p(M-1)\rfloor +1}^{M-1}\binom{M-2}{k-1}p^{k-1} (1-p)^{M-2-(k-1)}\\
=& (M-1) \P\left(S_{M-2}\geq \lfloor 2p(M-1)\rfloor  \right)
\leq  (M-1) \P\left(S_{M-2}\geq \frac{3}{2}p(M-2)\right).
\end{align*}
We used the bound $\lfloor 2p(M-1)\rfloor > \frac{3}{2}p(M-2)$, which is a consequence of the fact that $p(M-1)$ goes to infinity. This will be shown at the end of the proof.
Then, by using \eqref{corollary1}, with $\epsilon=\frac{p}{2}$ we obtain:
$$\sum_{k=\lfloor 2p(M-1)\rfloor +1}^{M-1} k\binom{M-1}{k} r(k) p^{k-1}(1-p)^{M-1-k} \leq (M-1) \exp\left( -\frac{p(M-2)}{12}\right).
$$
We consider now the first part. Since:
$$\frac{1}{k}( e^{N(1-2\r)-2} -(M-1-k) e^{\beta N})= \frac{1}{k}( e^{N(1-2\r)-2} - e^{ (\alpha +\beta) N} ) + e^{\beta N}
$$
we clearly see that $r(k)$ is increasing in $k$ if $\alpha+\beta<1-2\varrho$ is fulfilled. This condition will be proven later on. %\rot{Sven: Changed the sentence above. Mete: Added a small line so that it really becomes explicit} 
Thus:
\begin{align*}
\sum_{k=0}^{\lfloor 2p(M-1)\rfloor }& \binom{M-1}{k} kr(k) p^{k-1}(1-p)^{M-1-k}\\
&\leq \frac{1}{p}\max_{k\in \{0,\ldots, \lfloor 2p(M-1)\rfloor \}} kr(k)\\
& \leq2(M-1) \cdot r( 2p(M-1)).
\end{align*}
Let us examine this last term: since $p< e^{-N I(\beta)}$, we observe
\begin{align*}
r( 2p(M-1)) =& \P\Bigl( \exp(\langle \xi^2 | \wt \xi^1 \rangle) > \frac{1}{2p}\bigl( e^{N(1-2\r-\alpha)-2} - e^{\beta N}\bigr)+e^{\beta N}\Bigr)\\
\leq & \P\Bigl( \exp(\langle \xi^2 | \wt \xi^1 \rangle) > \frac{1}{2}\bigl( e^{N(1-2\r-\alpha+I(\beta))-2} - e^{(\beta+I(\beta)) N}\bigr) \Bigr).
\end{align*}
First of all let us show that $1-2\r-\alpha+I(\beta)>\beta+I(\beta)$, so that the first term dominates the second term:
Indeed, this is equivalent to proving that $\a + \beta < 1-2\r$. By concavity we have for $x\in(0,1)$:
$$ I(x)\leq \log \left( \frac{(1+x)^2}{2} +\frac{(1-x)^2}{2}\right)= \log ( 1+x^2) \leq x^2 \leq x.
$$
From this we obtain
$$\a +\beta <\a +\beta_o = \a +I(\a) \leq 2\a \leq I(1-2\r) \leq 1-2\r.
$$
This proves that $1-2\r-\alpha+I(\beta)>\beta+I(\beta)$ and also concludes the statement that $r(k)$ is increasing (see above). 
Now take $\gamma$ such that $1-2\r-\epsilon<\gamma<1-2\r-\frac{\epsilon}{2}$ for an $\epsilon > 0$. Then
$$1-2\r-\epsilon<\gamma<1-2\r-\alpha+I(\beta).
$$
For $N$ sufficiently large we get:
$$r( 2p(M-1)) \leq \P\left( \exp(\langle \xi^2 | \wt \xi^1 \rangle) > e^{\gamma N}\right).
$$
By applying \eqref{corollary2} and for $N$ sufficiently large one sees that:
$$r( 2p(M-1)) \leq \P\left( \exp(\langle \xi^2 | \wt \xi^1 \rangle) > e^{\gamma N}\right) \leq \exp\left( -N I(\gamma)\right).
$$
Now if we bring everything together, we finally arrive at:
\begin{align*}
\P & \Bigl(\exists \mu \; \exists i  : T_i\bigl(\wt \xi^\mu \bigr) \neq \xi^\mu_i \Bigr) \\
&\leq N \cdot M \cdot \P\Bigl(\sum\limits_{\mu=2}^M \exp\bigl( -\xi^\mu_i \xi^1_i + \sum\limits_{j\neq i} \xi^\mu_j \wt \xi^1_j \bigr) - \exp\bigl( \xi^\mu_i \xi^1_i + \sum\limits_{j\neq i} \xi^\mu_j \wt \xi^1_j \bigr) > ce^{N(1-2\r)} \Bigr)\\
&\leq N\cdot M\Bigl( 2(M-1)\exp\bigl( -NI(\gamma)\bigr) + (M-1) \exp\bigl( -\frac{p(M-2)}{12}\bigr)\Bigr)\\
&\leq 2N\frac{M}{M-1}\exp\left(-N(I(\gamma)-2\a)\right) +N\frac{M}{M-1}\exp\Bigl( 2\a-\frac{M-2}{M-1}\cdot\frac{p(M-1)}{12}\Bigr).
\end{align*}
But by the definition of $\gamma$, we have $I(\gamma)>2\a$. So the first term clearly tends to $0$.
For the second term by using the lower bound on $p$, we have that:
$$p(M-1)\geq \exp( N(\a-I(\beta)-\eta)).
$$
But we know that $I(\beta)<I(\beta_o)=\a$ and $\eta = \frac{1}{2} (\alpha - I(\beta))$, so $\a-I(\beta)-\eta>0$. Therefore the second term converges also to $0$. Also a straightforward consequence of this bound is that $p(M-1)$ goes to infinity and thus establishes the still open fact that $\lfloor 2p(M-1)\rfloor > \frac{3}{2}p(M-2)$, used above.
So clearly if the condition $\alpha < \frac{I(1-2\r)}{2}$ is fulfilled, we obtain
$$\P  \left( \exists \mu \; \exists i  : T_i\left(\wt \xi^\mu \right) \neq \xi^\mu_i \right) \to 0
$$
This finishes the proof.

\end{proof}

\vspace{0.5cm}
{\bf Acknowledgement}:
The final publication is available at Springer via
\newline http://dx.doi.org/10.1007/s10955-017-1806-y

%\bibliographystyle{aps-nameyear}      % American Physical Society (APS) style, author-year citations
\bibliographystyle{abbrv}

\bibliography{LiteraturDatenbank}                % name your BibTeX data base

\begin{thebibliography}{10}

\bibitem{AGS}
D.~J. Amit, H.~Gutfreund, and H.~Sompolinsky.
\newblock Spin-glass models of neural networks.
\newblock {\em Phys. Rev. A (3)}, 32(2):1007--1018, 1985.

\bibitem{AGS2}
D.~J. Amit, H.~Gutfreund, and H.~Sompolinsky.
\newblock Storing infinite numbers of patterns in a spin-glass model of neural
  networks.
\newblock {\em Phys. Rev. Lett.}, 55:1530--1533, Sep 1985.

\bibitem{Bov98}
A.~Bovier.
\newblock Sharp upper bounds on perfect retrieval in the {H}opfield model.
\newblock {\em J. Appl. Probab.}, 36(3):941--950, 1999.

\bibitem{BovierNiederhauser}
A.~Bovier and B.~Niederhauser.
\newblock The spin-glass phase-transition in the {H}opfield model with
  {$p$}-spin interactions.
\newblock {\em Adv. Theor. Math. Phys.}, 5(6):1001--1046, 2001.

\bibitem{dembozeitouni}
A.~Dembo and O.~Zeitouni.
\newblock {\em Large deviations techniques and applications}, volume~38 of {\em
  Stochastic Modelling and Applied Probability}.
\newblock Springer-Verlag, Berlin, 2010.
\newblock Corrected reprint of the second (1998) edition.

\bibitem{Hopfield1982}
J.~J. Hopfield.
\newblock Neural networks and physical systems with emergent collective
  computational abilities.
\newblock {\em Proc. Nat. Acad. Sci. U.S.A.}, 79(8):2554--2558, 1982.

\bibitem{KrotovHopfield2016}
D.~Krotov and J.~J. Hopfield.
\newblock Dense associative memory for pattern recognition.
\newblock In D.~D. Lee, M.~Sugiyama, U.~V. Luxburg, I.~Guyon, and R.~Garnett,
  editors, {\em Advances in Neural Information Processing Systems 29}, pages
  1172--1180. Curran Associates, Inc., 2016.

\bibitem{loukianova}
D.~Loukianova.
\newblock Lower bounds on the restitution error in the {H}opfield model.
\newblock {\em Probab. Theory Related Fields}, 107(2):161--176, 1997.

\bibitem{Lo98}
M.~L{\"o}we.
\newblock On the storage capacity of {H}opfield models with correlated
  patterns.
\newblock {\em Ann. Appl. Probab.}, 8(4):1216--1250, 1998.

\bibitem{Lo99a}
M.~L{\"o}we.
\newblock The storage capacity of generalized {H}opfield models with
  semantically correlated patterns.
\newblock {\em Markov Process. Related Fields}, 5(1):1--19, 1999.

\bibitem{LV05}
M.~L{\"o}we and F.~Vermet.
\newblock The storage capacity of the {H}opfield model and moderate deviations.
\newblock {\em Statist. Probab. Lett.}, 75(4):237--248, 2005.

\bibitem{LV07}
M.~L{\"o}we and F.~Vermet.
\newblock The capacity of {$q$}-state {P}otts neural networks with parallel
  retrieval dynamics.
\newblock {\em Statist. Probab. Lett.}, 77(14):1505--1514, 2007.

\bibitem{MPRV}
R.~J. McEliece, E.~C. Posner, E.~R. Rodemich, and S.~S. Venkatesh.
\newblock The capacity of the {H}opfield associative memory.
\newblock {\em IEEE Trans. Inform. Theory}, 33(4):461--482, 1987.

\bibitem{Newman_hopfield}
C.~M. Newman.
\newblock Memory capacity in neural network models: Rigorous lower bounds.
\newblock {\em Neural Networks}, 1(3):223--238, 1988.

\bibitem{talagrand}
M.~Talagrand.
\newblock Rigorous results for the {H}opfield model with many patterns.
\newblock {\em Probab. Theory Related Fields}, 110(2):177--276, 1998.

\end{thebibliography}
%\nocite{*}

\end{document}